\title{\vspace{-0.6cm} Two extremal problems on intersecting families}

\documentclass[11pt]{article}

\usepackage{amsmath, amssymb, amsthm, url, tikz, verbatim}

\newtheorem{theorem}{Theorem}[section]

\newtheorem{lemma}[theorem]{Lemma}

\newtheorem{conjecture}[theorem]{Conjecture}

\newcommand{\dive}{{\textup{div}}}

\usetikzlibrary{calc}
\usetikzlibrary{patterns}
\usetikzlibrary{decorations.markings}
\usetikzlibrary{arrows,shapes.geometric,%
	decorations.pathreplacing,shapes,shadows}

\date{}

\author{
Hao Huang \thanks{
Department of Math and CS, Emory University, Atlanta, GA 30322, USA.
Email: hao.huang@emory.edu. Research supported in part by the Collaboration Grants from the Simons Foundation.
}
}

\begin{document}
\maketitle
\abstract
In this short note, we address two problems in extremal set theory regarding intersecting families. The first problem is a question posed by Kupavskii: is it true that given two disjoint cross-intersecting families $\mathcal{A}, \mathcal{B} \subset \binom{[n]}{k}$, they must satisfy $\min\{|\mathcal{A}|, |\mathcal{B}|\} \le \frac{1}{2} \binom{n-1}{k-1}$? We give an affirmative answer for $n \ge 2k^2$, and construct families showing that this range is essentially the best one could hope for, up to a constant factor. The second problem is a conjecture of Frankl. It states that for $n \ge 3k$, the maximum diversity of an intersecting family $\mathcal{F} \subset \binom{[n]}{k}$ is equal to $\binom{n-3}{k-2}$. We are able to find a construction beating the conjectured bound for $n$ slightly larger than $3k$, which also disproves a conjecture of Kupavskii.
% Introduction

\section{Disjoint cross-intersecting families}
One of the most famous results in extremal set theory is the Erd\H os-Ko-Rado Theorem \cite{erdos-ko-rado}: for $n \ge 2k$, an intersecting family $\mathcal{F} \subset \binom{[n]}{k}$ has size at most $\binom{n-1}{k-1}$. The Erd\H os-Ko-Rado Theorem has many analogues and generalizations. One particularly interesting generalization is by considering two families instead of one. We say two families $\mathcal{A}$ and $\mathcal{B}$ are {\it cross-intersecting}, if for every $A \in \mathcal{A}$ and $B \in \mathcal{B}$, $A \cap B \neq \emptyset$. Pyber \cite{pyber} showed that when $n$ is large in $k, l$, for $\mathcal{A} \subset \binom{[n]}{k}$, $\mathcal{B} \subset \binom{[n]}{l}$, we have $|\mathcal{A}||\mathcal{B}| \le \binom{n-1}{k-1}\binom{n-1}{l-1}$. Later the same inequality for a precise range $n \ge 2\max\{k, l\}$ was established by Matsumoto and Tokushige \cite{m-t}. The Erd\H os-Ko-Rado Theorem follows immediately by setting $k=l$ and $\mathcal{A}=\mathcal{B}$. 

 Recently Kupavskii \cite{kupavskii_cross} asked the following question: given two cross-intersecting families $\mathcal{A}$ and $\mathcal{B}$ that are disjoint, is it true that 
$$\min \left\{|\mathcal{A}|, |\mathcal{B}| \right\}\} \le \frac{1}{2} \binom{n-1}{k-1}?$$
This bound, if true, is clearly tight. This is because we can always split the extremal example in Erd\H os-Ko-Rado Theorem, i.e.  a $1$-star $\mathcal{S}$, into two subfamilies $\mathcal{S}_1, \mathcal{S}_2$ as evenly as possible. Then $\mathcal{A}=\mathcal{S}_1$ and $\mathcal{B}=\mathcal{S}_2$ are cross-intersecting and disjoint, and each has about $\frac{1}{2}\binom{n-1}{k-1}$ subsets. In this section, we give a positive answer to this question for $n$ quadratic in $k$.

\begin{theorem} \label{thm_cross}
For $n \ge 2k^2$, given two disjoint cross-intersecting families $\mathcal{A}, \mathcal{B} \subset \binom{[n]}{k}$, we have
$$\min\{|\mathcal{A}|, |\mathcal{B}|\} \le \frac{1}{2} \binom{n-1}{k-1}.$$
\end{theorem}

As a warm-up, first we show that when $n$ is at least cubic in $k$, this statement is true. Consider a pair of disjoint crossing-intersecting families $\mathcal{A}$ and $\mathcal{B}$ of $k$-sets of $[n]$. If both $\mathcal{A}$ and $\mathcal{B}$ are intersecting, then $\mathcal{A} \cup \mathcal{B}$ is also intersecting, by the Erd\H os-Ko-Rado Theorem, for $n \ge 2k$, we have
$$|\mathcal{A}|+|\mathcal{B}| \le \binom{n-1}{k-1},$$
and thus we have the desired inequality 
$$\min \{|\mathcal{A}|, |\mathcal{B}|\} \le \frac{1}{2}\binom{n-1}{k-1}.$$
Now suppose at least one of $\mathcal{A}$ and $\mathcal{B}$ is not intersecting, without loss of generality we may assume that $\mathcal{A}$ is not intersecting, then there exists $A_1, A_2 \in \mathcal{A}$, such that $A_1 \cap A_2 = \emptyset$. Now the number of sets that intersect with both $A_1$ and $A_2$ provides an upper bound for $|\mathcal{B}|$, which is at most
\begin{align*}
k^2 \binom{n-2}{k-2} = \frac{k^2(k-1)}{(n-1)} \binom{n-1}{k-1} < \frac{1}{2} \binom{n-1}{k-1}
\end{align*}
when $n \ge 2k^3$.\\

Next we will improve the range to $n \ge 2k^2$. The main tool used in this proof is the technique of shifting, which allows us to limit our attention to sets with certain structure. In this section we will only state and prove some relevant results. For more background on the applications of shifting in extremal set theory, we refer the reader to the survey \cite{shifting_survey} by Frankl. 

\begin{proof}[Proof of Theorem \ref{thm_cross}]
Now assume $n \ge 2k^2$, suppose there exist disjoint cross-intersecting families $\mathcal{A}, \mathcal{B} \subset \binom{[n]}{k}$ such that $$\min\{|\mathcal{A}|, |\mathcal{B}|\}> \frac{1}{2} \binom{n-1}{k-1}.$$
We will prove the following statement: given positive integer $k, l$, and $n \ge k+l$, suppose $\mathcal{A} \subset \binom{[n]}{k}$, $\mathcal{B} \subset \binom{[n]}{l}$ are cross-intersecting. If 
$|\mathcal{A}| >\max\{k, l\} \binom{n-2}{k-2}$ and $|\mathcal{B}|> \max\{k,l\} \binom{n-2}{l-2}$,
then there exists some element $x$ contained in every subset of $\mathcal{A}$ and $\mathcal{B}$. Assuming this claim, if there exists $x$ such that $\mathcal{A}, \mathcal{B}$ are subfamilies of the $1$-star centered at $x$, then $|\mathcal{A} \cup \mathcal{B}| \le \binom{n-1}{k-1}$, and  Theorem \ref{thm_cross} follows from the disjointness of $\mathcal{A}$ and $\mathcal{B}$. Otherwise, either $|\mathcal{A}|$ or $|\mathcal{B}|$ has to be strictly smaller than 
$$k\binom{n-2}{k-2} = k \cdot\frac{k-1}{n-1} \binom{n-2}{k-2} \le \frac{1}{2}\binom{n-2}{k-2}$$ 
for $n \ge 2k^2$, which also proves Theorem \ref{thm_cross}. 

To show the claim, we use induction on $k, l, n$. Given a family $\mathcal{F}$, define  the $(i, j)$-shifting $S_{ij}$ as follows: let
$$S_{ij}(\mathcal{F})=\{S_{ij}(F): F \in \mathcal{F}\},$$
where 
$$S_{ij}(F)= \begin{cases}
F' & i \in F,~j \not\in F,~F'=F \setminus \{i\} \cup \{j\} \not\in \mathcal{F};\\
F  & \textup{otherwise}.
\end{cases}
.$$
It is well-known that if we apply $S_{ij}$ on $\mathcal{A}$ and $ \mathcal{B}$ simultaneously, the resulting families are still cross-intersecting. Therefore we can iteratively apply the shifting $S_{ij}$ for $j>i$ until we reach stable families ($S_{ij}(\mathcal{A})=\mathcal{A}$, $S_{ij}(\mathcal{B})=\mathcal{B}$). We claim that not only $A \in \mathcal{A}$ and $B \in \mathcal{B}$ must intersect, their intersection must also contain some element from $\{1, \cdots, k+l\}$. Suppose not, consider all the pairs $(A, B)$ such that $A \cap B \cap [k+l] =\emptyset$, pick the one which minimizes $|A \cap B|$. Since $A \cap B \neq \emptyset$, there exists $i \in \{k+l+1, \cdots, n\}$ such that $i \in A \cap B$, and also $j \in [k+l]$ such that $j \not\in A \cup B$. Since $S_{ij}(A)=A$ an $S_{ij}(B)=B$, we must have $A'=A \setminus \{i\} \cup \{j\} \in \mathcal{A}$. Note that $A' \cap B \cap [k+l]$ is still empty, and $|A' \cap B|<|A \cap B|$, contradicting the choice of $A, B$.

Let $\mathcal{K}=\{A \cap [k+l]: A \in \mathcal{A}\}$ and $\mathcal{L}=\{B \cap [k+l]: B \in \mathcal{B}\}$. We denote by $\mathcal{K}_i$ and $\mathcal{L}_i$ the family of $i$-subsets in $\mathcal{K}$ and $\mathcal{L}$ respectively. Every set $A \in \mathcal{A}$ must instersect $[k+l]$ with a subset from $\mathcal{K}$, therefore
$$|\mathcal{A}| \le \sum_{i=1}^k |\mathcal{K}_i| \binom{n-k-l}{k-i},$$
Similarly, $$|\mathcal{B}| \le \sum_{i=1}^l |\mathcal{L}_i| \binom{n-k-l}{l-i}.$$
Recall that $|\mathcal{A}| > \max\{k, l\} \binom{n-2}{k-2}$, since 
$$\sum_{i=1}^k \binom{n-k-l}{k-i} \binom{k+l-2}{i-2} = \binom{n-2}{k-2},$$
we know that there exists $i \in [k]$, such that $|\mathcal{K}_i| > \max\{k, l \} \binom{k+l-2}{i-2} \ge i \binom{k+l-2}{i-2}$. Similarly there exists $j$ such that $|\mathcal{L}_j| > \max\{k, l\}\binom{k+l-2}{j-2} \ge j \binom{k+l-2}{j-2}$. Note that $\mathcal{K}_i$ and $\mathcal{L}_j$ are cross-intersecting and $k+l \ge i+j$, so by induction, as long as $(n, k, l) \neq (k+l, i, j)$, there exists $x$ in every set of $\mathcal{K}_i$ and $\mathcal{L}_j$. Suppose there exists $B \in \mathcal{B}$, such that $x \not\in B$, then every set in $\mathcal{K}_i$ must intersect $B$ and contains $x$, there are less than $l \binom{n-2}{i-2} \le \max\{k, l\}\binom{n-2}{i-2}$ such subsets, contracting that $\mathcal{K}_i$ is large. Similarly we can show that $x$ is contained in every set of $\mathcal{A}$. When $(n, k, l)=(k+l, i, j)$, we know that $\mathcal{A} \subset \binom{[k+l]}{k}$ and $\mathcal{B} \subset \binom{[k+l]}{l}$. Since $\mathcal{A}$ and $\mathcal{B}$ are cross-intersecting, $\mathcal{A}$ and $\mathcal{B}^c=\{[k+l]\setminus B: B \in \mathcal{B}\}$ must be disjoint. Therefore $|\mathcal{A}|+|\mathcal{B}| \le  \binom{k+l}{k}$. However, suppose $k\ge l$, then by the assumption,
$$|\mathcal{A}|+|\mathcal{B}| \ge k \binom{k+l-2}{k-2}+ k\binom{k+l-2}{l-2}>\binom{k+l}{k}.$$
The last inequality is true for all $k \ge 2, l \ge 2$ except $(k,l)=(2, 2)$. In this case it is easy to check the statement is also true. This completes the proof of the previous claim, as well as Theorem \ref{thm_cross}.
\end{proof}

Naturally one would wonder whether then range $n \ge 2k^2$ could be further improved. We will show that a quadratic range is necessary here. We consider the following construction, let $t \ge 2$ be a fixed integer much smaller than $k$ or $n$. We choose $\mathcal{A}$ so that it consists of all the subsets whose intersection with $[t+1]$ is either $\{1\}$ or $\{2, \cdots, t+1\}$. $\mathcal{B}$ consists of all the subsets whose intersection with $[t+1]$ contains $1$ and some element of $\{2, \cdots, t+1\}$. Then 
$$|\mathcal{A}| = \binom{n-t-1}{k-1} + \binom{n-t-1}{k-t},$$
$$|\mathcal{B}| = \binom{n-1}{k-1}-\binom{n-t-1}{k-1}.$$

We will choose $k/n$ to be a fixed constant, and let $n, k$ tends to infinity. To estimate the sizes of $\mathcal{A}$ and $\mathcal{B}$, the following lemma is useful.
\begin{lemma}\label{lemma_approx}
	Suppose $k=\alpha n$ for some fixed $\alpha \in (0, 1)$. Then for fixed $t, i$, when $n \rightarrow \infty$, 
	$$\binom{n-t}{k-i}/\binom{n}{k} \rightarrow \alpha^i (1-\alpha)^{t-i} .$$
\end{lemma}
\begin{proof}
	\begin{align*}
	\dfrac{\binom{n-t}{k-i}}{\binom{n}{k}} &=\frac{k\cdots (k-i+1)\cdot (n-k)\cdots (n-k-(t-i)+1)}{n \cdots (n-t+1)}\\
	&=\frac{k}{n}\cdots \frac{k-i+1}{n-i+1} \cdot \frac{n-k}{n-i} \cdots \frac{n-k-(t-i)+1}{n-t+1} \rightarrow \alpha^{i} (1-\alpha)^{t-i}.
	\end{align*}
\end{proof}

%Observe that when $k/n \sim \alpha$ and $n, k \rightarrow \infty$, by Lemma \ref{lemma_approx},
%$$\frac{|\mathcal{A}|}{\binom{n-1}{k-1}} \sim (1-\alpha)^t + \alpha^{t-1}(1-\alpha).$$
%$$\frac{|\mathcal{B}|}{\binom{n-1}{k-1}} \sim 1-(1-\alpha)^t.$$
%\begin{align*}
%\frac{|\mathcal{B}|}{\binom{n-1}{k-1}}=1-\frac{{n-t-1 \choose k-1}}{{n-1 \choose k-1}}=1-\frac{(n-k-1) \cdots (n-k-t+1)}{(n-1) \cdots (n-t)}
%\end{align*}

Now we returning to our construction. First observe that $|\mathcal{A}| + |\mathcal{B}| > \binom{n-1}{k-1}$. Moreover when $k/n \sim \alpha$ and $n, k \rightarrow \infty$, by Lemma \ref{lemma_approx},
$$\frac{|\mathcal{A}|}{\binom{n-1}{k-1}} \rightarrow (1-\alpha)^t + \alpha^{t-1}(1-\alpha).$$

$$\frac{|\mathcal{B}|}{\binom{n-1}{k-1}} \rightarrow 1-(1-\alpha)^t.$$

Note that for fixed $t$, there exists some small positive constant $c$, such that for $\alpha$ in a small interval $[1-(1/2)^{1/t}, 1-(1/2)^{1/t}+c]$, both $1-(1-\alpha)^t>1/2$ and $(1-\alpha)^t+\alpha^{t-1}(1-\alpha)>1/2$. This shows that for $\alpha$ in this interval, if we let $k=\alpha n$ and $n, k$ are sufficiently large, there exists disjoint cross-intersecting $\mathcal{A}, \mathcal{B}$, both of size strictly greater than $\frac{1}{2}\binom{n-1}{k-1}$. Note that $1-(1/2)^{1/t}$ tends to $0$ when $t$ goes to infinity. Therefore it is not possible to prove Theorem \ref{thm_cross} for $n>Ck$ for any fixed constant $C$. 

As pointed out to us by Frankl and Kupavskii \cite{private}, the same construction actually shows the quadratic range in Theorem \ref{thm_cross} is best possible, up to a constant factor. This can be seen by setting $t=k-1$ in the construction. Then $|\mathcal{A}|=\binom{n-k}{k-1}+(n-k)$ and $|\mathcal{B}|=\binom{n-1}{k-1}-\binom{n-k}{k-1}$. As long as $|\mathcal{B}|>|\mathcal{A}|$ (which is true until $n>Ck^2$ for some constant $C$), one can move subsets from $\mathcal{B}$ to $\mathcal{A}$ and still have a cross-intersecting family, since $\mathcal{B}$ itself is intersecting. Recall that the sum of their sizes is strictly greater than $\binom{n-1}{k-1}$, therefore Theorem \ref{thm_cross} is only correct for $n$ at least qudratic in $k$. Frankl and Kupavskii have also obtained results similar to Theorem \ref{thm_cross} using different methods.

The proof of Theorem \ref{thm_cross} relies on the assumption that $n$ is quadratic in $k$, when we compare $k \binom{n-2}{k-2}$ with $\frac{1}{2}\binom{n-1}{k-1}$. One may wonder whether it is possible to show that for smaller $n$, either $|\mathcal{A}|$ or $|\mathcal{B}|$ cannot exceed $c \binom{n-1}{k-1}$, for some constant $c$ strictly smaller than $1$. The following result confirms this speculations, and actually implies that for $n=Ck$, $\min\{|\mathcal{A}|, |\mathcal{B}|\} \le (\frac{1}{2}+\delta_C) \binom{n-1}{k-1},$  where $\delta_C$ goes to zero as $C$ tends to infinity. 

\begin{theorem}\label{thm_weaker}
	For $n \ge 2k$, given two cross-intersecting families $\mathcal{A}, \mathcal{B} \subset \binom{[n]}{k}$ that are disjoint, we have
	$$\min\{|\mathcal{A}|,|\mathcal{B}|\} \le \frac{1}{2}\binom{n-1}{k-1} \cdot \frac{n-2}{n-k-1}.$$
\end{theorem}
\begin{proof}
The key tool that will be used this proof is the spectrum of Kneser graphs. The Kneser graph $KG(n, k)$ has vertices corresponding to all the $k$-subsets of $[n]$, and two vertices are adjacent if and only if the two corresponding sets are disjoint. It is known that (see for example on Page $200$ of \cite{godsil-royle})
 its adjacency matrix $M$ has eigenvalues $(-1)^{i+1}\binom{n-k-i}{k-i}$ of multiplicity $\binom{n}{i}-\binom{n}{i-1}$, where $i=0, \cdots, k$, and $\binom{n}{-1}$ is defined as $0$. Moreover, the all-one vector is an eigenvector of the largest eigenvalue $\binom{n-k}{k}$. We denote by $\lambda_1, \cdots, \lambda_{n \choose k}$ the eigenvalues in the aforementioned order, i.e. non-increasing in absolute value, and assume that their corresponding eigenvectors are ${\bf v}_1, \cdots, {\bf v}_{\binom{n}{k}}$, which form an orthonormal basis.

We consider the characteristic vectors ${\bf 1}_{\mathcal{A}}$ and ${\bf 1}_{\mathcal{B}}$ of the two disjoint cross-intersecting families. Both of them are in the space $\{0, 1\}^{{n \choose k}}$. We can express them as linear combinations of the eigenvectors:
$${\bf 1}_{\mathcal{A}} = \alpha_1{\bf v}_1 + \cdots + \alpha_{\binom{n}{k}} {\bf v}_{n \choose k},$$
$${\bf 1}_{\mathcal{B}} = \beta_1{\bf v}_1 + \cdots + \beta_{\binom{n}{k}} {\bf v}_{n \choose k}.$$
Since $\mathcal{A}$ and $\mathcal{B}$ are disjoint, the inner product of ${\bf 1}_{\mathcal{A}}$ and ${\bf 1}_{\mathcal{B}}$ equals $0$. This gives 
\begin{align} \label{eqn_1}
\alpha_1\beta_1 +\cdots+ \alpha_{\binom{n}{k}}\beta_{\binom{n}{k}}=0. 
\end{align}
Moreover, from the cross-intersecting property, we have 
\begin{align} \label{eqn_2}
0=\langle {\bf 1_{\mathcal{A}}}, M\cdot {\bf 1_{\mathcal{B}}} \rangle = \lambda_1\alpha_1\beta_1 +\cdots+ \lambda_{n\choose k}\alpha_{\binom{n}{k}}\beta_{\binom{n}{k}}
\end{align}
Let 
$$K=\frac{1}{2}\left(\binom{n-k-1}{k-1}-\binom{n-k-2}{k-2}\right).$$

We multiply $K$ to the \eqref{eqn_1} and add the resulting identity to the \eqref{eqn_2}. It is not hard to observe that the coefficient of $\alpha_1\beta_1$ is equal to $K+\lambda_1$, and the rest of the coefficients have absolute value at most 
$$\frac{1}{2}\left(\binom{n-k-1}{k-1}+\binom{n-k-2}{k-2}\right):=L
.$$
Therefore, by the triangle inequality,
\begin{align} \label{ineq_3}
|(K+\lambda_1) \alpha_1\beta_1| &= |\sum_{i=2}^{n \choose k} (\lambda_i+K) \alpha_i\beta_i|\le L \cdot \sum_{i=2}^{n \choose k} |\alpha_i||\beta_i| \nonumber\\
& \le L \cdot (\sum_{i=2}^{n \choose k} \alpha_i^2)^{1/2} \cdot (\sum_{i=2}^{n \choose k} \beta_i^2)^{1/2}
\end{align}
Recall that ${\bf 1}/\sqrt{\binom{n}{k}}$ is a unit eigenvector for the eigenvalue $\lambda_1$. Note that $|\mathcal{A}|=\|{\bf 1}_{\mathcal{A}}\|^2 = \sum_{i=1}^{n \choose k}\alpha_i^2$. On the other hand $|\mathcal{A}|=\langle {\bf 1}, {\bf 1}_{\mathcal{A}} \rangle$, therefore
$$\sqrt{\binom{n}{k}}\alpha_1 = \alpha_1^2+ \cdots + \alpha_{n\choose k}^2$$
We have a similar inequality for $\{\beta_i\}$. Plugging both of them into the previous inequality \eqref{ineq_3}, we have
\begin{align*}
\frac{K+\lambda_1}{L} \alpha_1^2 \cdot \frac{K+\lambda_1}{L} \beta_1^2 
\le (\sqrt{n \choose k}\alpha_1 -\alpha_1^2)(\sqrt{n \choose k}\beta_1 -\beta_1^2)
\end{align*}
Therefore either 
$$\frac{K+\lambda_1}{L} \alpha_1^2 \le \sqrt{n \choose k}\alpha_1 -\alpha_1^2,$$
or a similar inequality holds for $\beta$. Solving this inequality, we get
\begin{align*}
|\mathcal{A}|&=\sqrt{n \choose k} \alpha_1 \le \binom{n}{k} \cdot \frac{L}{K+L+\lambda_1}\\
&=\binom{n}{k} \cdot \frac{1}{2} \cdot \frac{\binom{n-k-1}{k-1}+\binom{n-k-2}{k-2}}{\binom{n-k}{k}+\binom{n-k-1}{k-1}}\\
&= \frac{1}{2}\binom{n-1}{k-1} \cdot \frac{n-2}{n-k-1}.
\end{align*}
\end{proof}

We do not know whether it is possible to further improve this upper bound, say for $2k+1 \le n \le 2k^2$. Observe that for example when $n=2k$, Theorem \ref{thm_weaker} gives $\min\{|\mathcal{A}|, |\mathcal{B}|\} \le \binom{2k-1}{k-1}$. This bound is best possible. This can be seen by pairing each $k$-set with its complement, and partitioning the $\binom{2k-1}{k-1}$ pairs into $\mathcal{A}$ and $\mathcal{B}$ as evenly as possible. However, even for $n=2k+1$ this bound does not seem to be sharp. It would be great if for every value of $n$, the maximum of $\min\{|\mathcal{A}|, |\mathcal{B}|\}$ can be determined precisely.

\section{Diversity of intersecting families}
Given an intersecting family $\mathcal{F}$ of $k$-subsets of $[n]$. Its {\it diversity}, denoted by $\dive(\mathcal{F})$, is defined as the number of sets not passing through the most popular element. For example, the $1$-star extremal construction in the Erd\H os-Ko-Rado Theorem has diversity $0$, since every set contains the center of the $1$-star. The Hilton-Milner Theorem \cite{hilton-milner} is equivalent to finding the maximum size of an intersecting family with diversity at least $1$. It is natural to ask the following question: given a family $\mathcal{F}$ of $k$-subset of $[n]$, what is the maximum diversity it can possibly have? Let $\mathcal{F}_x=\{F: x \in F \in \mathcal{F}\}$, the goal is to maximize $\dive(\mathcal{F})=|\mathcal{F}|- \max_{x \in [n]} |\mathcal{F}_x|$.

This question was first suggested by Katona and addressed by Lemons and Palmer \cite{lemons-palmer}. They showed that for $n>6k^3$, the diversity of $\mathcal{F}$ is at most $\binom{n-3}{k-2}$, with the equality attained by the following family:
$$\mathcal{F}=\left\{F: F \in \binom{[n]}{k}, |F \cap [3]| \ge 2 \right\}.$$
Recently, Frankl \cite{frankl_diversity} proved that $\dive(\mathcal{F}) \le \binom{n-3}{k-2}$ for $n \ge 6k^2$, and conjectured that the same holds for $n>3k$. More recently, Kupavskii \cite{kupavskii_diversity} verified Frankl's conjecture for $n>Ck$, for some large constant $C$. He also consider the intersecting families 
$$\mathcal{D}_r=\{D: D \in \binom{[n]}{k}, |D \cap [2r+1]| \ge r+1\}.$$ 
Here the ``two out of three'' family $\mathcal{F}$ is just $\mathcal{D}_1$. By computing the diversities of $\mathcal{D}_r$ for $r=1, \cdots, k-1$, it is not hard to show that $\mathcal{D}_r$ has the largest diversity among all $\mathcal{D}_1, \cdots \mathcal{D}_{k-1}$, for $(k-1)(2+\frac{1}{r})+1 \le n \le (k-1)(2+\frac{1}{r-1})+1$. This observation prompts the following stronger conjecture in \cite{kupavskii_diversity}:
\begin{conjecture}
Fix $n\ge 2k$, and consider an intersecting family $\mathcal{F} \subset \binom{[n]}{k}$. If for some $r \in \mathbb{Z}_{\ge 0}$, we have $(k-1)(2+\frac{1}{r})+1 \le n \le (k-1)(2+\frac{1}{r-1})+1$, then $\dive(\mathcal{F}) \le \dive(\mathcal{D}_r)$.
\end{conjecture}
Note that the $r=1$ case corresponds to Frankl's conjecture. Below we will present a construction showing that $3k$ is not the right threshold for Frankl's conjecture, thus disproving both conjectures.

Let $t$ be a positive integer, and $\mathcal{H}$ be an intersecting family of subsets of $[t]$, which is not necessarily uniform. Let
$$\mathcal{F}=\{F: F \in \binom{[n]}{k}, F \cap [t] \in \mathcal{H}\},$$ 
then $\mathcal{F}$ is also intersecting. Denote by $N_i$ the number of $i$-sets in $\mathcal{H}$, and $N_i(x)$ the number of $i$-sets in $\mathcal{H}$ not containing the element $x$. It is not hard to see that for $x \in [t]$,
$$|\mathcal{F} \setminus \mathcal{F}_x|=\sum_{i=1}^t N_i(x) \binom{n-t}{k-i},$$
and for $x \in [n] \setminus [t]$, 
$$|\mathcal{F}\setminus \mathcal{F}_x|=\sum_{i=1}^t N_i \binom{n-t-1}{k-i}.$$

\begin{theorem}
For $k$ sufficiently large and  $3k < n < (2+\sqrt{3})k$, there exists a family $\mathcal{F}$ such that 
$$\dive(\mathcal{F}) > \binom{n-3}{k-2}.$$ 
\end{theorem}
\begin{proof}
Let $t=6$, and 
$$\mathcal{G}=\{\{1,2,3\}, \{2,3,4\}, \{3,4,5\}, \{4,5,1\}, \{5,1,2\}, \\
\{1,3,6\}, \{2,4,6\}, \{2,5,6\}, \{3,5,6\}, \{1,4,6\}\}.$$
This family of $10$ sets of size $3$ is intersecting, every pair of elements appears exactly twice, and every element exactly $5$ times. Define
$$\mathcal{H}=\{F: F \subset [t], \textup{~there~exists~}G \in \mathcal{G} \textup{~such~that~} G \subset F \}.$$
Now we can compute $N_i$ and $N_i(x)$ for $\mathcal{H}$. We have $N_3=10$, $N_4=15$, $N_5=6$, $N_6=1$. And for each $x \in [6]$, $N_3(x)=5$, $N_4(x)=5$, $N_5(x)=1$, $N_6(x)=0$. Therefore, when we assume $k=\alpha n$ for fixed constant $\alpha$, as $n, k$ tends to infinity, by Lemma \ref{lemma_approx}, for $x \in [6]$,
\begin{align*}
\frac{|\mathcal{F} \setminus \mathcal{F}_x|}{{n \choose k}} &= \frac{\sum_{i=1}^6 N_i(x) \binom{n-t}{k-i}}{\binom{n}{k}} \rightarrow \sum_{i=1}^6 N_i(x)\cdot  \alpha^i (1-\alpha)^{t-i}.\\
&=5\alpha^3(1-\alpha)^3+5\alpha^4(1-\alpha)^2+\alpha^5(1-\alpha)\\
&=5\alpha^3-10\alpha^4 +6 \alpha^5 - \alpha^6 :=f_1(\alpha).
\end{align*}
For $x \not \in [6]$, 
\begin{align*}
\frac{|\mathcal{F} \setminus \mathcal{F}_x|}{{n \choose k}} &= \frac{\sum_{i=1}^6 N_i \binom{n-t-1}{k-i}}{\binom{n}{k}} \rightarrow \sum_{i=1}^6 N_i \alpha^i (1-\alpha)^{t+1-i}\\
&=10\alpha^3-25\alpha^4+21\alpha^5-6\alpha^6 :=f_2(\alpha).
\end{align*}
From Lemma \ref{lemma_approx}, we also have that $\binom{n-3}{k-2} \rightarrow \binom{n}{k} \cdot \alpha^2(1-\alpha)$. Solving $f_1(\alpha) >\alpha^2(1-\alpha)$, we have $\alpha \in (2-\sqrt{3}, 1)$. Solving $f_2(\alpha)>\alpha^2(1-\alpha)$, we have $\alpha  \in (\frac{9-\sqrt{57}}{12}, 1)$. Combining these two ranges, we know that when $\alpha \in (2-\sqrt{3}, 1)$, $k =\alpha n$ and $n$ goes to infinity, for every $x$, 
$$|\mathcal{F}\setminus \mathcal{F}_x| > \binom{n-2}{k-3},$$
thus the diversity of $\mathcal{F}$ is strictly greater than $\binom{n-2}{k-3}$.
\end{proof}

One can check that that the family $\mathcal{H}$ used in this construction is a maximum intersecting family of subsets of $[6]$. Moreover it also has the largest diversity. We believe that this property is the main reason that causes the resulting uniform family $\mathcal{F}$ to have large diversity. For $\mathcal{D}_r$, the family playing the role of $\mathcal{H}$ consists of subsets of $[2k+1]$ of size at least $k+1$. In order to completely settle the problem of determining the maximum diversity of a uniform intersecting family for every $n$, perhaps one should first prove the non-uniform version of the diversity problem: given an integer $n$, what is the family $\mathcal{F} \subset 2^{[n]}$ that has the maximum diversity? Motivated by the above discussions, the following conjecture seems natural. Let $n=2k+1$, and $\mathcal{Q}_k=\{A: A \subset [2k+1], |A| \ge k+1\}$.
\begin{conjecture}\label{conj_odd}
	For $n =2k+1$, suppose $\mathcal{F} \subset 2^{[n]}$ is intersecting. Then  $$\dive(\mathcal{F}) \le \dive(\mathcal{Q}_k)=\sum_{i=k+1}^{2k} \binom{2k}{i}.$$
\end{conjecture}
When $n=2k$, the situation could be slightly more complicated. Ideally the extremal family $\mathcal{F}$ should contain all the subsets of size at least $k+1$, together with half of the $k$-sets. Note that in this case, $\mathcal{F}$ is intersecting if and only if $\mathcal{F}_k$, its subfamily consisting of all the $k$-sets, is intersecting. So to maximize $\dive(\mathcal{F})$, we need to look for $\mathcal{F}_k \subset \binom{[2k]}{k}$ having the largest diversity. By the Erd\H os-Ko-Rado theorem, $|\mathcal{F}_k| \le \binom{2k-1}{k-1}$. And therefore 
$$\dive(\mathcal{F}_k) \le |\mathcal{F}_k|-|\mathcal{F}_k| \cdot \frac{k}{2k} \le \frac{1}{2} \binom{2k-1}{k-1}.$$

However this bound can only be attained when a regular $k$-uniform intersecting family of subset of $[2k]$ of size $\binom{2k-1}{k-1}$ exists. Brace and Daykin \cite{brace-daykin} showed that this happens if and only if $k$ is not a power of $2$. 
When $k$ is a power of $2$, Ihringer and Kupavskii \cite{ki} showed that the maximum size of such a regular family is $\binom{2k-1}{k-1}-3$. It is plausible that for even $n=2k$, $\dive(\mathcal{F})$ is always maximized by $\mathcal{F}_k \cup \binom{[n]}{\ge k+1}$, where $\mathcal{F}_k$ is a $k$-uniform family of size $\binom{2k-1}{k-1}$ as regular as possible. This prompts the following conjecture for the even case.
\begin{conjecture}\label{conj_even}
For $n=2k$, suppose $\mathcal{F} \subset 2^{[n]}$ is an intersecting family. If $k$ is not a power of $2$, then 
$$\dive(\mathcal{F}) \le \frac{1}{2}\binom{2k-1}{k-1}+\binom{2k-1}{k+1} + \cdots + \binom{2k-1}{2k-1};$$
and if $k$ is a power of $2$, then
$$\dive(\mathcal{F}) \le \frac{1}{2}\left(\binom{2k-1}{k-1}-1\right)+\binom{2k-1}{k+1} + \cdots + \binom{2k-1}{2k-1}.$$
\end{conjecture}
The validity of Conjecture \ref{conj_odd} and \ref{conj_even} has been checked using a computer for $n \le 6$.
~\\

\noindent \textbf{Acknowledgment. }The author would like to thank Peter Frankl and Andrey Kupavskii for their helpful comments and observation regarding the sharpness of Theorem \ref{thm_weaker}, and bringing up the reference \cite{ki}; and Jie Han for his significant contribution at the early stage of research.

\end{document}